\documentclass[11pt]{amsproc}
\usepackage[a4paper]{geometry}
\usepackage{amssymb}

\usepackage[sc,small]{caption}
\usepackage{graphics,overpic}

\def\phi{\varphi}

\def\R{{\mathbb{R}}}
\def\Z{{\mathbb{Z}}}
\def\Bar{\overline}
\def\wt{\widetilde}
\def\id{\mathord{\rm id}}

\def\allowhyphenation{\penalty1000\hskip0pt}
\def\dash{\discretionary{-}{}{-}\allowhyphenation}
\def\lput(#1,#2)#3{\put(#1,#2){\hbox to 0pt{\hss{#3}}}}
\def\cput(#1,#2)#3{\put(#1,#2){\hbox to 0pt{\hss{#3}\hss}}}

\newtheoremstyle{mein}{\medskipamount}{\medskipamount}{\it}{0pt}{\sc}{.}{1ex}{}
\theoremstyle{mein}
\newtheorem{prop}{Proposition}
\newtheorem{lem}[prop]{Lemma}
\newtheorem{thm}[prop]{Theorem}
\newtheorem{cor}[prop]{Corollary}
\newtheorem{defn}[prop]{Definition}

\newenvironment{proof}{\begin{list}{}
    {\leftmargin0pt\labelwidth0pt\itemindent\labelsep}\item[{\sc Proof.}]}
{\qed\end{list}}
\newtheoremstyle{rem}{\medskipamount}{\medskipamount}{\rm}{0pt}{\it}{.}{1ex}{}
\theoremstyle{rem}
\newtheorem{remark}{Remark}
\newtheorem{example}{Example}

\newenvironment{mlist}{\setcounter{enumi}{0}\begin{list}
    {\stepcounter{enumi}\arabic{enumi}.\hfill}
    {\labelsep0.2ex\leftmargin0ex
    \itemsep0pt\parsep0pt\labelwidth2.5ex\itemindent6ex}}
{\end{list}\smallskip}

\begin{document}
\title [A curvature theory for discrete surfaces]
    {A curvature theory for discrete surfaces
	\\ based on mesh parallelity}
\def\Email#1{email {\spaceskip0pt plus 0.5pt\xspaceskip\spaceskip #1}}
\author[A.\ I.\ Bobenko]{Alexander I.\ Bobenko}
\address{Alexander Bobenko. Institut f\"ur Mathematik, TU Berlin, Strasse
des 17.\ Juni 136, D~10623 Berlin. \Email{bobenko @math.
tu-berlin. de}.
%Tel.\ (+49) 30 314 24655, Fax (+49) 30 314 79282.
}
\author[H.\ Pottmann]{Helmut Pottmann}
\address{Helmut Pottmann.
Geometric Modeling and Industrial Geometry, TU Wien,
Wiedner Hauptstr.\ 8--10/104,
A~1040 Wien.
\Email{pottmann @geometrie. tuwien. ac. at}.
%Tel.\ (+43) 1 58801 11310. Fax (+43) 1 58801 11399.
}

\author[J. Wallner]{Johannes Wallner}
\address{Johannes Wallner. Institut f\"ur Geometrie, TU Graz,
Kopernikusgasse 24, A~8010 Graz. \Email{j. wallner @tugraz. at}.
%Tel.\ (+43) 316 873 8440. Fax (+43) 316 873 8448.
}

 \begin{abstract} We consider a general theory of curvatures of discrete 
surfaces equipped with edgewise parallel Gauss images, and where mean and 
Gaussian curvatures of faces are derived from the faces' areas and mixed 
areas. Remarkably these notions are capable of unifying notable previously
defined classes of surfaces, such as discrete isothermic minimal surfaces and
surfaces of constant mean curvature. We discuss various 
types of natural Gauss images, the existence of principal curvatures, 
constant curvature surfaces, Christoffel duality, Koenigs nets, contact 
element nets, s-isothermic nets, and interesting special cases such as 
discrete Delaunay surfaces derived from elliptic billiards.
 \end{abstract}

\maketitle

\section{Introduction}

A new field of {\em discrete differential geometry} is presently emerging 
on the border between differential and discrete geometry; see, for 
instance, the recent books \cite{Bobenko_Schroeder_Sullivan_Ziegler_2008, 
bobenko-2008-ddg}. Whereas classical differential geometry investigates 
smooth geometric shapes (such as surfaces), and discrete geometry studies 
geometric shapes with a finite number of elements (such as polyhedra), 
discrete differential geometry aims at the development of discrete 
equivalents of notions and methods of smooth surface theory. The latter 
appears as a limit of refinement of the discretization. Current progress 
in this field is to a large extent stimulated by its relevance for 
applications in computer graphics, visualization and architectural design.

Curvature is a central notion of classical differential geometry, and 
various discrete analogues of curvatures of surfaces have been studied. A 
well known discrete analogue of the Gaussian curvature for general 
polyhedral surfaces is the angle defect at a vertex. One of the most 
natural discretizations of the mean curvature of simplicial surfaces 
(triangular meshes) introduced in \cite{Pinkall_Polthier_1993} is based on 
a discretization of the Laplace\dash Beltrami operator (cotangent 
formula).

Discrete surfaces with quadrilateral faces can be treated as discrete 
parametrized surfaces. There is a part of classical differential geometry 
dealing with parametrized surfaces, which goes back to Darboux, Bianchi, 
Eisenhart and others. Nowadays one associates this part of differential 
geometry with the theory of integrable systems; see \cite{Fordy_Wood_1994, 
Rogers_Schief_2002}. Recent progress in discrete differential geometry has 
led not only to the discretization of a large body of classical results, 
but also, somewhat unexpectedly, to a better understanding of some 
fundamental structures at the very basis of the classical differential 
geometry and of the theory of integrable systems; see 
\cite{bobenko-2008-ddg}.

This point of view allows one to introduce natural classes of surfaces 
with constant curvatures by discretizing some of their characteristic 
properties, closely related to their descriptions as integrable systems. 
In particular, the discrete surfaces with constant negative Gaussian 
curvature of \cite{Sauer_1950} and \cite{Wunderlich_1951} are discrete 
Chebyshev nets with planar vertex stars. The discrete minimal surfaces of 
\cite{bobenko-1996} are circular nets Christoffel dual to discrete 
isothermic nets in a two\dash sphere. The discrete constant mean curvature 
surfaces of \cite{BP99} and \cite{Hertrich-Jeromin_Hoffmann_Pinkall_1999} 
are isothermic circular nets with their Christoffel dual at constant 
distance. The discrete minimal surfaces of Koebe type in 
\cite{bobenko-2006-ms} are Christoffel duals of their Gauss images which 
are Koebe polyhedra. Although the classical theory of the corresponding 
smooth surfaces is based on the notion of a curvature, its discrete 
counterpart was missing until recently.

One can introduce curvatures of surfaces through the classical Steiner 
formula. Let us consider an infinitesimal neighborhood of a surface $m$ 
with the Gauss map $s$ (contained in the unit sphere $S^2$). For 
sufficiently small $t$ the formula
	$$
	m^t=m+ts
	$$
 defines smooth surfaces parallel to $m$. The infinitesimal area of the 
parallel surface $m^t$ turns out to be a quadratic polynomial of $t$ and 
is described by the Steiner formula
	\begin{equation}\label{eq:Steiner}
	dA(m^t)=(1-2Ht+Kt^2)\,dA(m),
	\end{equation}
 Here $dA$ is the infinitesimal area of the corresponding surface and $H$ 
and $K$ are the mean and the Gaussian curvatures of the surface $m$, 
respectively. In the framework of {\em relative} differential geometry 
this definition was generalized to the case of the Gauss map $s$ contained 
in a general convex surface.

A discrete version of this construction is of central importance for this 
paper. It relies on an edgewise parallel pair $m, s$ of polyhedral 
surfaces. It was first applied in \cite{Schief_2003, schief-2006} to 
introduce curvatures of circular surfaces with respect to arbitrary Gauss 
maps $s\in S^2$. We view $s$ as the Gauss image of $m$ and do not require 
it to lie in $S^2$, i.e., our generalization is in the spirit of relative 
differential geometry \cite{simon:1992}. Given such a pair, one has a 
one\dash parameter family $m^t=m+ts$ of polyhedral surfaces with parallel 
edges, where linear combinations are understood vertex\dash wise.

We have found an unexpected connection of the curvature theory to the 
theory of mixed volumes \cite{schneider-1993}. Curvatures of a pair 
$(m,s)$ derived from the Steiner formula are given in terms of the areas 
$A(m)$ and $A(s)$ of the faces of $m$ and $s$, and of their mixed area 
$A(m,s)$:
	\begin{eqnarray*}
	A(m^t)=(1-2 H t+K t^2)A(m),\quad
	H=-\frac{A(m,s)}{A(m)},\quad K=\frac{A(s)}{A(m)}.
	\end{eqnarray*}
 The mixed area can be treated as a scalar product in the space of 
polygons with parallel edges. The orthogonality condition with respect to 
this scalar product $A(m,s)=0$ naturally recovers the Christoffel 
dualities of \cite{bobenko-1996} and \cite{bobenko-2006-ms}, and discrete 
Koenigs nets (see \cite{bobenko-2008-ddg}). It is remarkable that the 
aforementioned definitions of various classes of discrete surfaces with 
constant curvatures follow as special instances of a more general concept 
of the curvature discussed in this paper.

It is worth to mention that the curvature theory presented in this paper 
originated in the context of multilayer constructions in architecture 
\cite{pottmann-2007-pm}.

 \section{Discrete surfaces and their Gauss images}
\label{sec:2}

This section sets up the basic definitions and our notation. It is 
convenient to use notation which keeps the abstract combinatorics of 
discrete surfaces separate from the actual locations of vertices. We 
consider a 2\dash dimensional cell complex $(V,E,F)$ which we refer to as 
{\em mesh combinatorics}. Any mapping $m:i\in V\mapsto m_i \in \R^3$ of 
the vertices to Euclidean space is called a {\em mesh}. If all vertices 
belonging to a face are mapped to co\dash planar points, we would like to 
call the mesh a {\em polyhedral surface}. If $f=(i_1,\dots, i_n)$ is a 
face with vertices $i_1,\dots, i_n$, we use the symbol $m(f)$ to denote 
the $n$-gon $m_{i_1},\dots, m_{i_n}$.

 \begin{defn} Meshes $m, m'$ having combinatorics $(V,E,F)$ are parallel, 
if for each edge $(i, j)\in E$, vectors $m_i-m_j$ and $m'_i-m'_j$ are 
linearly dependent.
 \end{defn}

Obviously for any given combinatorics there is a vector space $(\R^3)^V$ 
of meshes, and for each mesh there is a vector space of meshes parallel to 
$m$. If no zero edges $(i, j)$ with $m_i=m_j$ are present, parallelity is 
an equivalence relation. In case $m$ is a polyhedral surface without zero 
edges and $m'$ is parallel to $m$, then also $m'$ is a polyhedral surface, 
such that corresponding faces of $m$ and $m'$ lie in parallel planes.

A pair of parallel meshes $m, m'$ where corresponding vertices $m_i, m_i'$ 
do not coincide defines a system of lines $L_i=m_i\vee m_i'$. By 
parallelity, lines associated with adjacent vertices are co\dash planar, 
so the lines $L_i$ constitute a line congruence \cite{bobenko-2008-ddg}. 
It is easy to see that for simply connected combinatorics we can uniquely 
construct $m'$ from this congruence and a single seed vertex $m_{i_0}'\in 
L_{i_0}$, provided no faces degenerate and the lines $L_i$ intersect 
adjacent faces transversely.

A special case of this construction is a parallel pair $m, m'$ of 
polyhedral surfaces which are {\em offsets at constant distance $d$} of 
each other, in which case the lines $L_i$ are considered as {\em surface 
normals}. The vectors
	$$s_i = {1\over d}(m_i'-m_i)$$
 define the mesh $s$ called the {\em Gauss image} of $m$. Following 
\cite{pottmann-2007-pm, pottmann-2008-fg}, we list the three main 
definitions, or rather clarifications, of the otherwise rather vague 
notion of {\em offset}:

\begin{mlist}

\item[$*$\ ] {\it Vertex offsets:} the parallel mesh pair $m$, $m'$ is a 
vertex offset pair, if for each vertex $i\in V$, $\|m_i-m_i'\|=d$. The 
Gauss image $s$ is inscribed in the unit sphere $S^2$.

\item[$*$\ ] {\it Edge offsets:} $(m, m')$ is an edge offset pair, if 
corresponding edges $m_im_j$ and $m_i' m_j'$ are contained in parallel 
lines of distance $d$. The Gauss image $s$ in midscribed to the unit 
sphere (i.e., edges of $s$ are tangent to $S^2$ and $s$ is a {\em Koebe 
polyhedron}, see \cite{bobenko-2006-ms}).

\item [$*$\ ] {\it Face offsets:} $(m, m')$ is an face offset pair, if for 
each face $f\in F$, the $n$\dash gons $m(f)$, $m'(f)$ lie in parallel 
planes of distance $d$. The Gauss image $s$ is circumscribed to $S^2$. 
\end{mlist}

The polyhedral surfaces which possess face offsets are the {\em conical 
meshes}, where for each vertex the adjacent faces are tangent to a right 
circular cone. The polyhedral surfaces with quadrilateral faces which 
possess vertex offsets are the {\em circular surfaces}, i.e. their faces 
are inscribed in circles.

\begin{remark} Meshes which possess face offsets or edge offsets can be 
seen as entities of Laguerre geometry \cite{pottmann-2008-eolag}, while 
meshes with regular grid combinatorics which have vertex offsets or face 
offsets are entities of Lie sphere geometry \cite{bobenko-2007-org, 
bobenko-2008-ddg}.
 \end{remark}

\section{Areas and mixed areas of polygons}

As a preparation for the investigation of curvatures we study the area of 
$n$\dash gons in $\R^2$. We view the area as a quadratic form and consider 
the associated symmetric bilinear form. The latter is closely related to 
the well known {\em mixed area} of convex geometry.

\subsection{Mixed area of polygons.} \label{sec:area}

The oriented area of an $n$\dash gon $P=(p_0,\dots, p_{n-1})$ contained in 
a two\dash dimensional vector space $U$ is given by Leibniz' sector 
formula:
    \begin{equation}
    \label{eq:area}
    A(P)={1\over 2}\sum\nolimits_{0\le i<n}\det (p_i, p_{i+1}).
    \end{equation}
 Here and in the following indices in such sums are taken modulo $n$. The 
symbol {\it det} means a determinant form in $U$. Apparently $A(P)$ is a 
quadratic form in the vector space $U^n$, whose associated symmetric 
bilinear form is also denoted by the symbol $A(P,Q)$:
    \begin{align}
    \label{eq:lincomb}
    & A(\lambda P+\mu Q) =\lambda^2 A(P)+2 \lambda\mu A(P,Q)
        + \mu^2 A(Q).
    \end{align}
 Note that in Equation \eqref{eq:lincomb} the sum of polygons is defined 
vertex\dash wise, and that $A(P,Q)$ does not, in general, equal the well 
known mixed area functional. For a special class of polygons important in 
this paper, however, we have that equality.

 \begin{defn} We call two $n$\dash gons $P,Q\in U^n$ parallel if their
 corresponding edges are parallel.
 \end{defn}

 \begin{lem} \label{lem:convexmixed} If parallel $n$\dash gons $P,Q$ 
represent the positively oriented boundary cycles of convex polygons 
$K,L$, then \eqref{eq:lincomb} computes the mixed area of $K,L$. \end{lem}

 \begin{proof} For $\lambda,\mu\ge 0$, the polygon $\lambda P+\mu Q$ is 
the boundary of the domain $\lambda K+\mu L$, and so \eqref{eq:lincomb} 
immediately shows the identity of $A(P,Q)$ with the mixed area of $K,L$.
 \end{proof}

In view of Lemma \ref{lem:convexmixed}, we use the name {\em mixed area} 
for the symbol ``$A(P,Q)$'' in case polygons $P,Q$ are parallel. Next, we 
consider the {\em concatenation} of polygons $P_1,P_2$ which share a 
common sequence of boundary edges with opposite orientations which cancel 
upon concatenation. Successive concatenation of polygons $P_1,\dots,P_k$ 
is denoted by $P_1\oplus\ldots\oplus P_k$. It is obvious that 
$A(\bigoplus_i P_i)=\sum A(P_i)$, but also the oriented mixed areas of 
concatenations have a nice additivity property:

 \begin{lem} \label{lem:concat}
 Assume that $P_1\oplus\cdots\oplus P_k$ and $P_1'\oplus\cdots\oplus P_k'$ 
are two combinatorially equivalent concatenations of polygons, and that 
for $i=1,\dots, k$, polygons $P_i,P_i'$ are parallel. Then
    \begin{equation}
    A(P_1\oplus\cdots\oplus P_k,
	P_1'\oplus\cdots\oplus P_k')
     = A(P_1,P_1')+\cdots+A(P_k,P_k').
    \end{equation}
 \end{lem}

 \begin{proof} It is sufficient to consider the case $k=2$. We compute
    $
     A(P_1\oplus P_2,P_1'\oplus P_2')
    ={d\over dt}\big\vert_{t=0} A((P_1\oplus P_2)+t(P_1'\oplus P_2'))
     ={d\over dt}\big\vert_{t=0} A((P_1+tP_1')\oplus (P_2+tP_2'))
    = {d\over dt}\big\vert_{t=0} A(P_1+tP_1')
    + {d\over dt}\big\vert_{t=0} A(P_2+tP_2')
    =A(P_1,P_1')+A(P_2',P_2').
    $
 \end{proof}

\subsection{Signature of the area form.}

We still collect properties of the mixed area. This section is devoted to 
the zeros of the function $A(xP+yQ)$, where $P,Q$ are parallel $n$-gons in 
a 2\dash dimensional vector space $U$.

 \begin{figure}[t]
 \rightline{\begin{overpic}[width=\textwidth]{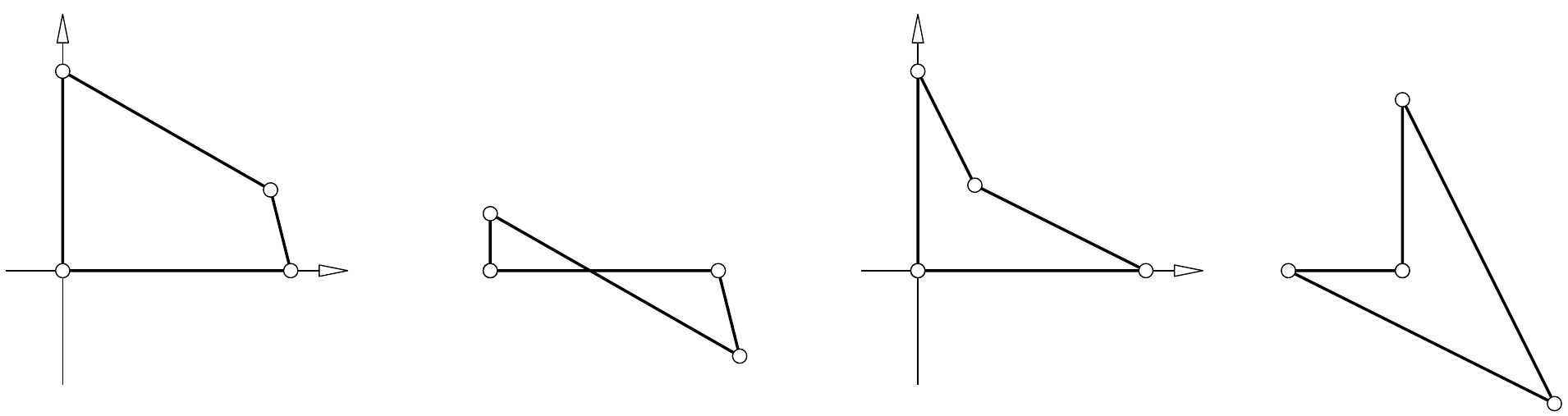}
 \cput(25,20){(a)}
 \cput(75,20){(b)}
 \lput(3,20){$p_0$}
 \lput(3,11){$p_1$}
 \put(17,6){$p_2$}
 \put(16,16){$p_3$}
 \lput(30,13){$p_0'$}
 \lput(30,9){$p_1'$}
 \put(46,11){$p_2'$}
 \put(48,3){$p_3'$}
 \end{overpic}}
 \caption{(a) Parallel quadrilaterals whose vertices lie
on the boundary of their convex hull. 
(b) Parallel quadrilaterals whose vertices do not lie on the boundary
of their convex hull.}
 \label{fig:signature}
 \end{figure}
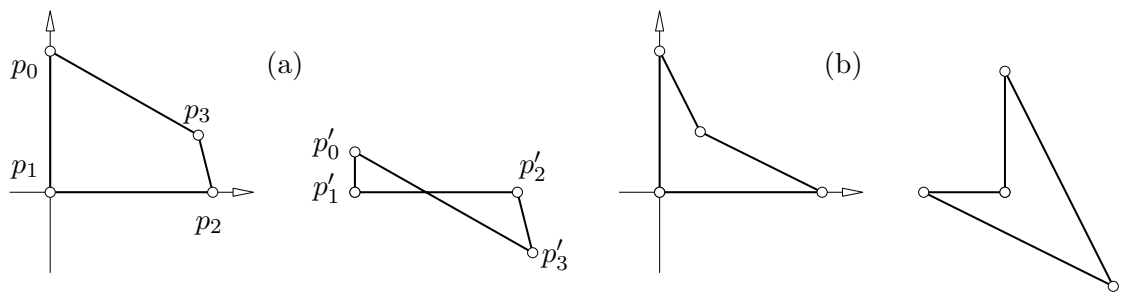

 \begin{thm} \label{thm:definite} Consider a quadrilateral $P$ which is 
nondegenerate, i.e., three consecutive vertices are never collinear. Then 
the area form in the space of quadrilaterals parallel to $P$ is indefinite 
if and only if all vertices $p_0,\dots, p_3$ are extremal points of their 
convex hull. If $P$ degenerates into a triangle, then the area form is 
semidefinite.
 \end{thm}

 \begin{proof} We choose an affine coordinate system such that $P$ has 
vertices ${0\choose 1}$, ${0\choose 0}$, ${1\choose 0}$, ${s\choose t}$ 
(cf.\ Figure \ref{fig:signature}). Translations have no influence on the 
area, so we restrict ourselves to computing the area of $Q$ parallel to 
$P$ with $q_1={0\choose 0}$, $q_3={s'\choose t'}$. Then $A(Q)= 
\mbox{\footnotesize$(s'\ t')$} \cdot \big({(t-1)/s \atop 1 }{1\atop 
(s-1)/t}\big) \cdot{s'\choose t'}$. The determinant of the form's matrix 
equals $(1-s-t)/st$, so the form is indefinite if and only if two or none 
of $s, t,1-s-t$ are negative, i.e., all vertices lie on the boundary of 
the convex hull.  In the degenerate case of three collinear vertices we 
compute areas of triangles all of which have the same orientation.
 \end{proof}

 \begin{prop} \label{prop:factorize} Assume that $n$-gons $P$, $Q$ are 
parallel but not related by a similarity transform. Consider the quadratic 
polynomial $\phi(x, y)=A(x P+y Q)$.
    \begin{mlist}

\item Suppose there is some combination $P'=\lambda P+\mu Q$ which is the 
vertex cycle of a strictly convex polygon $K$. Then $\phi$ factorizes and 
is not a square in $\R[x, y]$.

\item Assume that $n=4$ and that some combination $\lambda P+\mu Q$ is 
nondegenerate. Then $\phi$ is no square in $\R[x, y]$. It factorizes 
$\iff$ the vertices of $\lambda P+\mu Q$ are extremal points of their 
convex hull.

     \end{mlist}
     \end{prop}

     \begin{proof} 1. Change $(\lambda,\mu)$ slightly to 
$(\lambda',\mu')$, such that $|{\lambda\atop\mu}{\lambda'\atop\mu'}|\ne 0$ 
and $Q':=\lambda'P+\mu'Q$ still bounds a strictly convex polygon, denoted 
by $L$. Consider $\phi'(x, y)=A(x P'+yQ')$. As $\phi$ and $\phi'$ are 
related by a linear substitution of parameters, it is sufficient to study 
the factors of $\phi'$: According to \eqref{eq:lincomb}, the discriminant 
of $\phi'$ equals $4(A(P',Q')^2 -A(Q')A(P'))= 4(A(K,L)^2 -A(K)A(L))$, 
which is positive by Minkowski's inequality \cite{schneider-1993}. The 
statement follows.

In case 2 we observe that {\em any} element polygon parallel to $P$ arises 
from some $xP+yQ$ by a translation which does not change areas. It is 
therefore sufficient to consider the areas of the special quads treated in 
the proof of Theorem \ref{thm:definite}. The matrix of the area form which 
occurs there is denoted by $G$. Obviously $\phi$ factorizes $\iff$ $\det G 
\le 0$ $\iff$ the area form is indefinite or rank deficient. We see that 
$\det G\ne 0$, so rank deficiency does not occur (and consequently $\phi$ 
is no square). We use Theorem \ref{thm:definite} to conclude that $\phi$ 
factorizes $\iff$ the vertices of $\lambda P+\mu Q$ lie on the boundary of 
their convex hull.
 \end{proof}

\section{Curvatures of a parallel mesh pair}
\label{sec:curvature}

Our construction of curvatures for discrete surfaces is similar to the 
curvatures defined in relative differential geometry \cite{simon:1992}, 
which are derived from a field of `arbitrary' normal vectors. If the 
normal vectors employed are the usual Euclidean ones, then the curvatures, 
too, are the usual Euclidean curvatures.

A definition of curvatures which is transferable from the smooth to the 
discrete setting is the one via the change in surface area when we 
traverse a 1\dash parameter family of offset surfaces. Below we first 
review the smooth case, and afterwards proceed to discrete surfaces.

\subsection{Review of relative curvatures for smooth surfaces.}

Consider a smooth 2\dash dimensional surface $M$ in $\R^3$ which is 
equipped with a distinguished ``unit'' normal vector field $n:M\to \R^3$. 
It is required that for each tangent vector $v\in T_p M$, the vector 
$dn_p(v)$ is parallel to the tangent plane $T_p M$, so we may define a 
Weingarten mapping $\sigma_p:T_p M\to T_p M$ by $\sigma_p(v)=-dn_p(v)$ (a 
unit normal vector field in Euclidean space $\R^3$ fulfills this 
property). Then Gaussian curvature $K$ and mean curvature $H$ of the 
submanifold $M$ with respect to the normal vector field $n$ are defined as 
coefficients of $\sigma_p$'s characteristic polynomial
    \begin{equation}
    \chi_{\sigma_p}(\lambda,\mu):= \det(\lambda\id +\mu\sigma_p) =
    \lambda^2+2\lambda\mu H(p)+\mu^2 K(p).
    \end{equation}
 We consider an {\em offset surface} $M^\delta$, which is the image of $M$ 
under the offsetting map $e^\delta:p\mapsto p+\delta \cdot n(p)$. Clearly, 
tangent spaces in corresponding points of $M$ and $M^\delta$ are parallel, 
and corresponding surface area elements are related by
    \begin{equation}
    {dA^\delta\over dA~}\Big\vert_p=\det (de^\delta_p) =
    \det(\id+\delta \cdot dn)=
    \det(\id -\delta\sigma_p)=
	1-2\delta H + \delta^2 K,
    \label{eq:areasmooth}
    \end{equation}
 provided this ratio is positive. This equation has a direct analogue in 
the discrete case, which allows us to define curvatures for discrete 
surfaces.

\subsection{Curvatures in the discrete category.}

Let $m$ be a polyhedral surface with a parallel mesh $s$. We would like to 
think of $s$ as the Gauss image of $m$, but so far $s$ is arbitrary.  The 
meshes $m^\delta$ are {\em offsets} of $m$ at distance $\delta$ 
(constructed w.r.t.\ to the Gauss image mesh $s$). For each face $f\in F$, 
the $n$-gons $m(f)$, $s(f)$, and $m^\delta(f)$ lie in planes parallel to 
some two\dash dimensional subspace $U_{f}$. The area form in $U_f$ and the 
derived mixed area are both denoted by the symbol $A$. We have the 
following property:

 \begin{thm} If $m, s$ is a parallel mesh pair, then the area 
$A(m^\delta(f))$ of a face $f$ of an offset $m^\delta=m+\delta s$ obeys 
the law
    \begin{align}
    \label{eq:totalarea}
    & A (m^\delta(f))
    = (1 -2\delta H_f + \delta^2 K_f)
        A(m(f)), \quad \mbox{where}
    \\
    & H_f = -{A(m(f), s(f))\over A(m(f))}, \quad
    K_f = {A(s(f))\over A(m(f))}.
    \label{eq:gaussmean}
    \end{align}
    \end{thm}
    \begin{proof}
 Equation \eqref{eq:totalarea} can be shown face\dash wise and is then a 
direct consequence of \eqref{eq:lincomb}. As all determinant forms in a 
vector space are multiples of each other, neither $H_f$ nor $K_f$ depend 
on the choice of $A$.
 \end{proof}

Because of the analogy between Equations \eqref{eq:areasmooth} and
\eqref{eq:totalarea}, we define:

 \begin{defn} The functions $K_f,H_f$ of \eqref{eq:gaussmean} are the 
Gaussian and mean curvatures of the pair $(m,s)$, i.e. of the polyhedral 
surface $m$ with respect to the Gauss image $s$. They are associated to 
the faces of $m$.
 \end{defn}

Obviously, mean and Gaussian curvatures are only defined for faces of 
nonvanishing area. They are attached to the pair $(m, s)$ in an affine 
invariant way. There is a further obvious analogy between the smooth and 
the discrete cases: The Gauss curvature is the quotient of (infinitesimal) 
corresponding areas in the Gauss image and the original surface.

 \subsection{Existence of principal curvatures}

Similar to the smooth theory, we introduce principal curvatures 
$\kappa_1$, $\kappa_2$ of a face as the zeros of the quadratic polynomial 
$x^2-2Hx+K$, where $H$, $K$ are the mean and Gaussian curvatures. We shall 
see that in ``most'' cases that polynomial indeed factorizes, so principal 
curvatures exist. The precise statement is as follows:

 \begin{prop} \label{prop:quads} Consider a polyhedral surface $m$ with 
Gauss image $s$, and assume that for each face $f\in F$ mean and Gaussian 
curvatures $H_f$, $K_f$ are defined. Regarding the existence of principal 
curvatures $\kappa_{1, f}$ and $\kappa_{2, f}$, we have the following 
statements:
 \begin{mlist}

 \item For a quadrilateral $f$, $\kappa_{1, f}=\kappa_{2, f}$ $\iff$ 
$m(f)$, $s(f)$ are related by a similarity. If this is not the case, 
$\kappa_{i, f}$ exist $\iff$ the vertices of $m(f)$ or of $s(f)$ lie on 
the boundary of their convex hull.

 \item Suppose some linear combination of the $n$-gons $m(f)$, $s(f)$ is 
the boundary cycle of a strictly convex polygon. Then $\kappa_{i, f}$ 
exist, and $\kappa_{1, f}=\kappa_{2, f}$ $\iff$ $m(f)$ and $s(f)$ are 
related by a similarity transform.

\item Suppose $f$ is a quadrilateral and the Gauss image $s$ is inscribed 
in a strictly convex surface $\Sigma$. Then principal curvatures exist. 
They are equal if and only if $m(f)$ and $s(f)$ are related by a 
similarity transform.
  \end{mlist} \end{prop}

  \begin{proof} We consider the polynomial $\phi(x, y):=A(x\cdot 
m(f)+y\cdot s(f))$ as in Prop.\ \ref{prop:factorize}. The area of $m(f)$ 
is nonzero, otherwise curvatures are not defined. Thus, $\phi(x, y)$ is 
proportional to $\wt\phi(x, y):=x^2 - 2H_f x y + K_f y^2)$, and linear 
factors of $\phi$ correspond directly to linear factors of $g(x) 
:=\wt\phi(x,1)=x^2-2Hx+K$. So statements 1,2 follow directly from Prop.\ 
\ref{prop:factorize}. As to the third statement, note that the vertices of 
an $n$-gon which lie in a planar section of $\Sigma$ always are contained 
in the boundary of their convex hull, so we can apply 1.
 \end{proof}

\subsection{Edge curvatures}

In a smooth surface, a tangent vector $v\in T_p M$ indicates a {\em 
principal direction} with principal curvature $\kappa$, if and only if 
$-dn(v)=\kappa v$. For a discrete surface $m$ with combinatorics 
$(V,E,F)$, a tangent vector is replaced by an edge $(i, j)\in E$. By 
construction, edges $m_im_j$ are parallel to corresponding edges $s_is_j$ 
in the Gauss image mesh. We are therefore led to a {\em curvature 
$\kappa_e$ associated with the edge $e$}, which is defined by
	\begin{equation} e=(i,j)\in E 
	\implies
	s_j-s_i=\kappa_{i,j}(m_i-m_j) 
	\end{equation}
 (see Figure \ref{fig:edgecurv}). For a quad\dash dominant mesh this 
interpretation of all edges as principal curvature directions is 
consistent with the fact that discrete surface normals adjacent to an edge 
are co\dash planar \cite{pottmann-2008-fg}.

The newly constructed principal curvatures associated with edges are 
different from the previous ones, which are associated with faces. For a 
quadrilateral however, it is not difficult to relate the edge curvatures 
with the previously defined face curvatures:

\begin{prop} Consider a polyhedral surface $m$ with Gauss image $s$, and 
corresponding quadrilateral faces $m(f)=(m_0,\dots,m_3)$, 
$s(f)=(s_0,\dots,s_3)$. Then mean and Gaussian curvatures of that face are 
computable from its four edge curvatures by
    \begin{align}
        H_f
    &=
        {\kappa_{01} \kappa_{23} - \kappa_{12} \kappa_{30}
        \over
        \kappa_{01} + \kappa_{23} - \kappa_{12} - \kappa_{30}},
    \\
        K_f
    &=
        {
        \kappa_{01}\kappa_{12}\kappa_{23}\kappa_{30}
        \over
        \kappa_{01} + \kappa_{23} - \kappa_{12} - \kappa_{30}}
        \Big(
        {1\over \kappa_{12}}
        +{1\over \kappa_{30}}
        -{1\over \kappa_{01}}
         -{1\over \kappa_{23}}
        \Big)
    \end{align}
\end{prop}

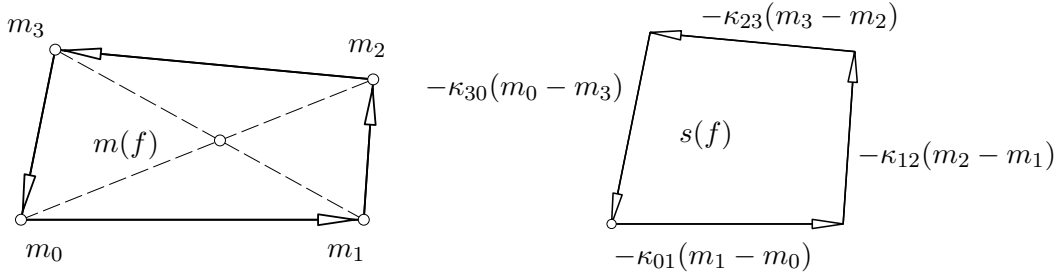
\begin{figure}[t]
	\medskip
	\centerline{\begin{overpic}[width=.45\columnwidth]{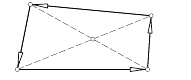}
        \cput(30,19){$m(f)$}
        \put(10,-2){$m_0$}
        \put(70,-2){$m_1$}
        \lput(14,43){$m_3$}
        \cput(78,39){$m_2$}
        \end{overpic}\hfill
   \begin{overpic}[width=.38\columnwidth]{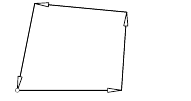}
        \put(25,25){$s(f)$}
        \put(30,53){$-\kappa_{23}(m_3-m_2)$}
        \cput(33,-4){$-\kappa_{01}(m_1-m_0)$}
        \put(67,20){$-\kappa_{12}(m_2-m_1)$}
        \lput(12,36){$-\kappa_{30}(m_0-m_3)$}
        \end{overpic}\hfill \vphantom{x}}
    \caption{Edge curvatures $\kappa_{i, i+1}$ associated with
a quadrilateral $m_0,\dots, m_3$ in a polyhedral surface $m$
with Gauss image $s$.}
    \label{fig:edgecurv}
\end{figure}

\begin{proof}
 We determine $\alpha_f$ such that $x^*:=(m_0\vee m_2)\cap(m_1\vee 
m_3)=(1-\alpha_f)m_1+\alpha_f m_3$. Likewise we determine $\beta_f$ such 
that $x^*=(1-\beta_f)m_2+\beta_f m_0$. The condition $\sum_{0\le i<4} 
\kappa_{i, k+1}(m_{i+1}-m_i) =0$ after some elementary manipulations leads 
to
    \begin{align*}
    H_f & = (1-\alpha_f) {\kappa_{23}+\kappa_{30}\over 2}
    +\alpha_f {\kappa_{01}+\kappa_{12}\over 2},
    \quad
    K_f  = (1-\alpha_f) \kappa_{23}\kappa_{30}
    + \alpha_f \kappa_{01}\kappa_{12}.
    \\
    H_f & = (1-\beta_f) {\kappa_{30}+\kappa_{01}\over 2}
    +\beta_f {\kappa_{12}+\kappa_{23}\over 2},
    \quad
    K_f  = (1-\beta_f) \kappa_{30}\kappa_{01}
    + \beta_f \kappa_{12}\kappa_{23}.
    \end{align*}
Equating the two expressions for $H_f$ and $K_f$  yields
the result.
 \end{proof}

\begin{remark} Using the line congruence $L_i=m_i\vee (m_i+s_i)$ (cf.\ 
Section \ref{sec:2}), for each edge $e=(i, j)$, we define a center of 
curvature associated with an edge $m_im_j$ as the point $c_e = L_i\cap 
L_j$. The familiar concept of curvature as the inverse distance of the 
center of curvature from the surface is reflected in the fact that the 
triangles $0s_is_j$ and $c_{e}m_im_j$ are transformed into each other by a 
similarity transformation with factor $1/\kappa_{e}$.
 \end{remark}

\section{Christoffel duality and discrete Koenigs nets}

We start with a general definition:

\begin{defn} \label{def:dual} Polyhedral surfaces $m, s$ are {\em 
Christoffel dual} to each other,
	$$s= m^*,
	$$
 if they are parallel, and their corresponding faces have vanishing mixed 
area (i.e., are orthogonal with respect to the corresponding bilinear 
symmetric form). Polyhedral surfaces possessing Christoffel dual are 
called {\em Koenigs nets}.
 \end{defn}

Duality is a symmetric relation, and obviously all meshes $s$ dual to $m$ 
form a linear space. In the special case of quadrilateral faces, duality 
is recognized by a simple geometric condition:

\begin{thm}\label{thm: dual=orthogonal}
 {\bf (Dual quadrilaterals via mixed area)} Two quadrilaterals $P=(p_1$, 
$p_2$, $p_3$, $ p_4)$ and $Q=(q_1$, $q_2$, $q_3$, $q_4)$ with parallel 
corresponding edges, $p_{i+1}-p_i\parallel q_{i+1}-q_i$, 
$i\in\mathbb{Z}\pmod 4$ are dual, i.e.,
	$$
	A(P,Q)=0
	$$
 if and only if their non\dash corresponding diagonals are parallel:
	$$
	(p_1p_3)\parallel(q_2q_4),\quad (p_2p_4)\parallel(q_1q_3).
	$$
 \end{thm}

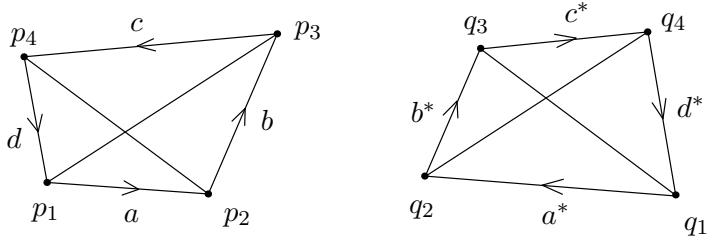
\begin{figure}[htbp]
\centering
	\setlength{\unitlength}{3947sp}
    \begin{picture}(4239,1492)(657,-1375)
	\put(657,-1375){\includegraphics[width=4250\unitlength]
		{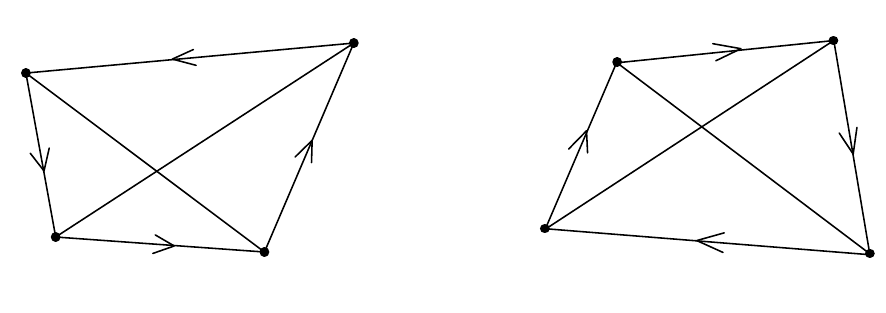}}
	\put(831,-1222){$p_1$}
	\put(2023,-1255){$p_2$}
	\put(2463,-84){$p_3$}
	\put(699,-127){$p_4$}
	\put(4881,-1311){$q_1$}
	\put(3175,-1202){$q_2$}
	\put(3513,-52){$q_3$}
	\put(4739,-30){$q_4$}
	\put(2253,-682){$b$}
	\put(1441,-61){$c$}
	\put(672,-795){$d$}
	\put(4140, -9){$c^*$}
	\put(4839,-600){$d^*$}
	\put(3998,-1271){$a^*$}
	\put(1400,-1259){$a$}
	\put(3184,-637){$b^*$}
    \end{picture}
\caption{Dual quadrilaterals.} \label{fig:dual_quads}
\end{figure}

 \begin{proof} Denote the edges of the quadrilaterals $P$ and $Q$ as in 
Figure \ref{fig:dual_quads}. For a quadrilateral $P$ with oriented edges 
$a,b,c,d$ we have
	$$
	A(P)=\frac{1}{2}([a,b]+[c+d]),
	$$
 where $[a,b]=\det(a,b)$ is the area form in the plane. The area of
the quadrilateral $P+tQ$ is given by
	$$
	A(P+tQ)=\frac{1}{2}([a+ta^*,b+tb^*]+[c+tc^*,d+td^*]).
	$$
 Identifying the linear terms in $t$ and using the identity
$a+b+c+d=0$, we get
	\begin{eqnarray*}
	4A(P,Q) & = & [a,b^*] + [a^*,b] + [c,d^*] + [c^*,d]\\
	 & = & [a+b,b^*] + [a^*,a+b] + [c+d,d^*] + [c^*,c+d]\\
	 & = & [a+b,b^*-a^*-d^*+c^*].
	\end{eqnarray*}
 Vanishing of the last expression is equivalent to the parallelism
of the non\dash corresponding diagonals, $(a+b)\parallel(b^*+c^*)$.
 \end{proof}

Theorem \ref{thm: dual=orthogonal} shows that for quadrilateral surfaces 
our definition of Koenigs nets is equivalent to the one originally 
suggested in \cite{bobenko-2007-koenigs, bobenko-2008-ddg}. For geometric 
properties of Koenigs nets we refer to these papers. It turns out that the 
class of Koenigs nets is invariant with respect to projective 
transformations.

\section{Polyhedral surfaces with constant curvature}\label{s:constant 
curvature}

Let $(m,s)$ be a polyhedral surface with its Gauss map as in Section 
\ref{sec:curvature}. We define special classes of surfaces as in classical 
surface theory, the only difference being the fact that the Gauss map is 
not determined by the surface. The treatment is similar to the approach of 
relative differential geometry.

\begin{def} \label{dfn:constant curvature surfaces} We say that a pair 
$(m,s)$ has constant mean (resp. Gaussian) curvature if the mean (resp. 
Gaussian) curvatures defined by (\ref{eq:gaussmean}) for all faces are 
equal. If the mean curvature vanishes identically, $H\equiv 0$, then the 
pair $(m,s)$ is called minimal. \end{def}

Although this definition refers to the Gauss map, the normalization of the 
length of $s$ is irrelevant, and the notion of constant curvature nets is 
well defined for discrete surfaces equipped with line congruences.

\begin{thm}\label{thm:minimal general} A pair $(m,s)$ is minimal if and 
only if $m$ is a discrete Koenigs net and $s$ is its Christoffel dual 
$s=m^*$. \end{thm}

\begin{proof} We have the equivalence $H=0 \ \iff \ A(m,s)=0 \ \iff \ 
s=m^*$. \end{proof}

\begin{figure}[t]
	\centering
	\includegraphics[width=0.7\textwidth]{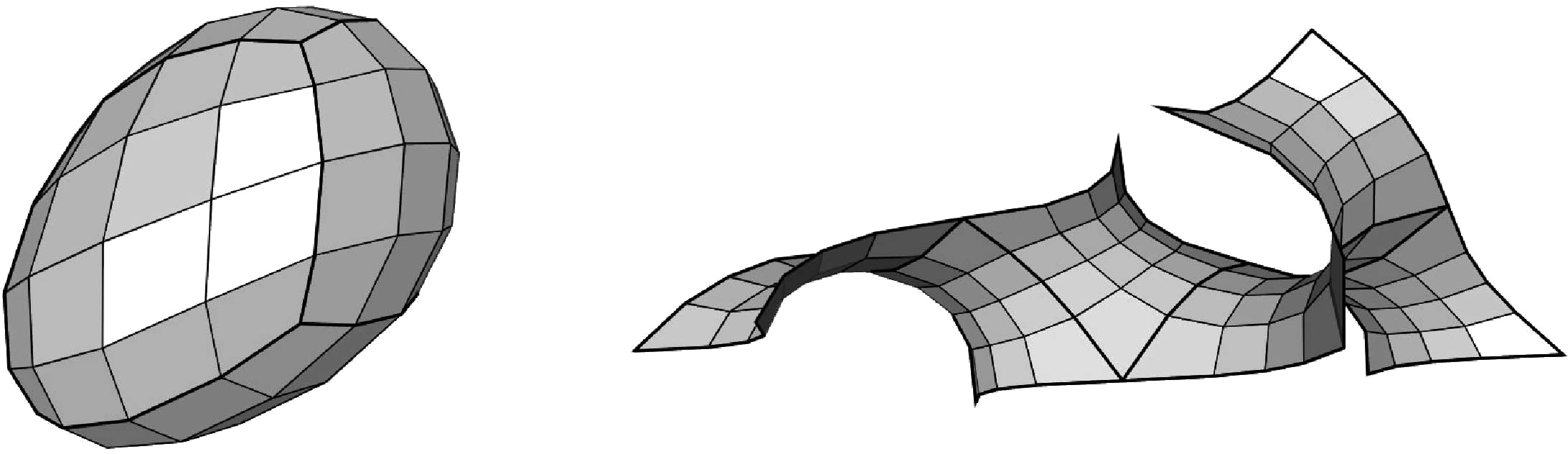}
\caption{Discrete Koenigs nets interpreted as a Gauss image $s$,
and its Christoffel dual minimal net $m=s^*$ (courtesy
P.\ Schr\"oder).}
\label{fig:minimal general}
\end{figure}

This result is analogous to the classical theorem of Christoffel 
\cite{christoffel} in the theory of smooth minimal surfaces. Figure 
\ref{fig:minimal general} presents an example of a discrete minimal 
surface $m$ constructed as the Christoffel dual of its Gauss image $s$, 
which is a discrete Koenigs net.

The statement about surfaces with nonvanishing constant mean curvature 
resembles the corresponding facts of the classical theory.

 \begin{thm}\label{thm:CMC general} A pair $(m,s)$ has constant mean 
curvature $H_0$ if and only if $m$ is a discrete Koenigs net and its 
parallel $m^{1/H_0}$ is the Christoffel dual of $m$:
	$$
	m^*=m+\frac{1}{H_0}s.
	$$
 The mean curvature of this parallel surface $(m+H_0^{-1}s,-s)$ (with the 
reversed Gauss map) is also constant and equal to $H_0$. The mid-surface 
$m+(2H_0)^{-1}s$ has constant positive Gaussian curvature $K_0=4H_0^2$ 
with respect to the same Gauss map $s$.
 \end{thm}
 \begin{proof}
 We have the equivalence
	$$
	A(m,s)=-H_0 A(m)\ \iff \
	A\Big(m,m+\frac{1}{H_0}s\Big)=0 \ \iff \
	m^*=m+\frac{1}{H_0}s.
	$$
For the Gaussian curvature of the mid-surface we get
	\begin{equation*}
	K_{\frac{1}{2H_0}}=\frac{A(s)}{A(m+\frac{1}{2H_0}s)} =
	\frac{A(s)}
	{A(m) + {1\over H_0} A(m,s) + ({1\over 2H_0})^2 A(s)} 
	= 4H_0^2.
	\end{equation*}
\end{proof}

It turns out that all surfaces parallel to a surface with constant 
curvature have remarkable curvature properties, in complete analogy to the 
classical surface theory. In particular they are linear Weingarten (For 
circular surfaces this was shown in \cite{schief-2006}).

 \begin{thm}\label{thm:Weingarten} Let $(m,s)$ be a polyhedral surface 
with constant mean curvature and its Gauss map. Consider the family of 
parallel surfaces $m^t=m+ts$. Then for any $t$ the pair $(m^t,s)$ is 
linear Weingarten, i.e., its mean and Gaussian curvatures $H_t$ and $K_t$ 
satisfy a linear relation
	\begin{equation}\label{eq:Weingarten}
	\alpha H_t+\beta K_t=1
	\end{equation}
 with constant coefficients $\alpha, \beta$.
 \end{thm}

 \begin{proof} Denote by $H$ and $K$ the curvatures of the basic surface 
$(m,s)$ with constant mean curvature. Let us compute the curvatures $H_t$ 
and $K_t$ of the parallel surface $(m+ts,s)$. We have
	\begin{align*}
	& \dfrac{A(m + (t+\delta)s )}{A(m + ts)}
	=\dfrac{1 - 2 H(t+\delta) + K(t+\delta)^2}{1 - 2Ht + Kt^2} 
	\\
	&=1 - 2 \delta \dfrac{H-Kt}{1-2Ht+Kt^2}+\delta^2
	\dfrac{K}{1-2Ht+Kt^2}= 1-2H_t\delta+K_t\delta^2.
	\end{align*}
 The last identity treats $m+(t+\delta)s$ as a parallel surface of
$m+ts$. Thus,
	$$
	H_t = \frac{H - Kt}{1 - 2Ht + Kt^2}, 
	\quad 
	K_t = \frac{K}{1 - 2Ht + Kt^2}.
	$$
 Note that $H$ is independent of the face, whereas $K$ is varying.
Therefore, with the above values for $H_t$ and $K_t$, relation
(\ref{eq:Weingarten}) is equivalent to
	$
	\frac{\alpha H}{1 - 2Ht} = \frac{\beta - \alpha t}{t^2} = 1,
	$
 which implies
	$$
	\alpha = \frac{1}{H} - 2t, \quad \beta = \frac{t}{H} - t^2.
	$$
\end{proof}

We see that any discrete Koenigs net $m$ can be extended to a minimal or 
to a constant mean curvature net by an appropriate choice of the Gauss map 
$s$. Indeed,
	\begin{itemize}
 \item[] $(m,s)$ is minimal for $s=m^*$;
 \item[] $(m,s)$ has constant mean curvature for $s=m^*-m$.
	\end{itemize}
 However, $s$ defined in such generality can lead us too far away from the 
smooth theory. It is natural to look for additional requirements which 
bring it closer to the Gauss map of a surface. These are exactly three 
cases of special Gauss images of Section \ref{sec:2}.

\subsection*{Cases with canonical Gauss image}

For a polyhedral surface $m$ which has a face offset $m'$ at distance 
$d>0$ (i.e., $m$ is a conical mesh) the Gauss image $s=(m'-m)/d$ is 
uniquely defined even without knowledge of $m'$, provided consistent 
orientation is possible. This is because $s$ is tangentially circumscribed 
to $S^2$ and there is only one way we can parallel translate the faces of 
$m$ such that they are in oriented contact with $S^2$. The same is true if 
$m$ has an edge offset, because an $n$\dash tuple of edges emanating from 
a vertex ($n\ge 3$) can be parallel translated in only one way so as to 
touch $S^2$.

It follows that for both cases a canonical Gauss image and canonical 
curvatures are defined. In case of an edge offset much more is known about 
the geometry of $s$. E.g.\ we can express the edge length of $s$ in terms 
of data read off from $m$ (see Figure \ref{fig:koebe}). The edges 
emanating from a vertex $s_i$ are contained in $s_i$'s tangent cone, which 
has some opening angle $\omega_j$. By parallelity of edges we can 
determine $\omega_j$ from the mesh $m$ alone. The ratio between edge 
length in the mesh and edge length in the Gauss image determines the 
curvature: $\kappa_{i,j} = \pm (\cot\omega_i+\cot\omega_j) /
 \|m_i-m_j\|$ (we skip discussion of the sign).

  \begin{figure}[t]
\centering
 \begin{overpic}[width=.30\textwidth]{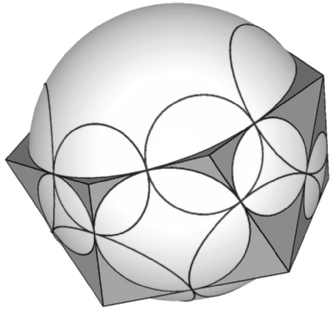}
    \lput(59,20){$c_i$}
    \put(50,71){$c_j$}
    \put(92,10){$s_i$}
 \end{overpic}
 \caption{A Koebe polyhedron $s$. The tangent cone from each vertex $s_i$ 
touches $S^2$ along a circle $c_i$. These circles form a packing, touching 
each other in the points where the edges touch $S^2$. It follows that the 
edge lengths are related to the opening angles $\omega_i$ of said cones: 
We have $\|s_i-s_j\|=\cot\omega_i+\cot\omega_j$.}
 \label{fig:koebe}
\end{figure}

\section{Curvature of principal contact element nets. Circular minimal and 
cmc surfaces}

\noindent In this section we are dealing with the case when the Gauss 
image $s$ lies in the two-sphere $S^2$, i.e., is of unit length, 
$\|s\|=1$. Our main example is the case of quadrilateral surfaces with 
regular combinatorics, called {\em Q-nets}. In this case a polyhedral 
surface $m$ with its parallel Gauss map $s$ is described by a map
	$$
	(m,s):\mathbb{Z}^2\to\mathbb{R}^3\times {S}^2.
	$$
 It can be canonically identified with a {\em contact element} net
	$$
	(m,{\mathcal P}):\mathbb{Z}^2\to\{ {\rm contact\ elements\ in}\
	\mathbb{R}^3\},
	$$
 where $\mathcal P$ is the oriented plane orthogonal $s$. We will call the 
pair $(m,s)$ also a contact element net. Recall that according to 
\cite{bobenko-2007-org} a contact element net is called {\em principal} if 
neighboring contact elements $(m,{\mathcal P})$ share a common touching 
sphere. This condition is equivalent to the existence of focal points for 
all elementary edges $(n,n')$ of the lattice $\mathbb{Z}^2\ni n,n'$, which 
are solutions to
	$$
	(m+ts)(n)=(m+ts)(n')
	$$
 for some $t$.

 \begin{thm}\label{thm:circular=contact element} Let 
$m:\mathbb{Z}^2\to\mathbb{R}^3$ be a Q-net with a parallel unit Gauss map 
$s:\mathbb{Z}^2\to S^2$. Then $m$ is circular, and $(m,s)$ is a principal 
contact element net. Conversely, for a principal contact element net 
$(m,s)$, the net $m$ is circular and $s$ is a parallel Gauss map of $m$.
 \end{thm}

 \begin{proof} The circularity of $m$ follows from the simple fact that 
any quadrilateral with edges parallel to the edges of a circular 
quadrilateral is also circular. Consider an elementary cube built by two 
parallel quadrilaterals of the nets $m$ and $m+s$. All the side faces of 
this cube are trapezoids, which implies that the contact element net 
$(m,s)$ is principal.
 \end{proof}

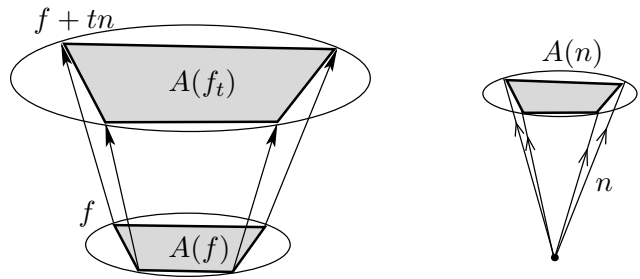
\begin{figure}[h]
	\centering
	\setlength{\unitlength}{4144sp}
	\begin{picture}(3715,1609)(364,-990)
	\put(364,-990){\includegraphics[width=3740\unitlength]
		{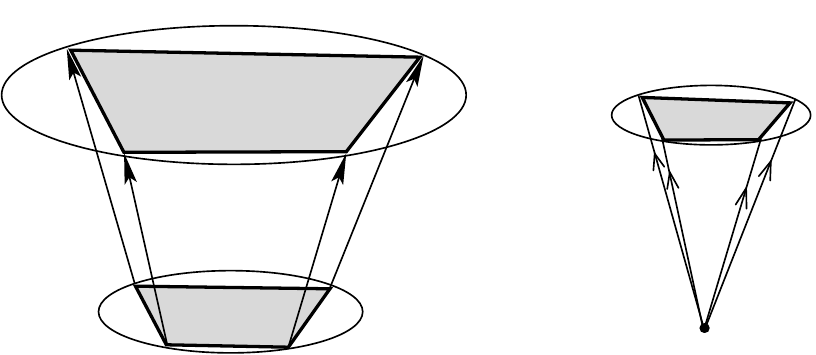}}
	\put(774,-655){$f$}
	\put(1306,-871){$A(f)$}
	\put(1306,119){$A(f_t)$}
	\put(507,496){$f+tn$}
	\put(3514,299){$A(n)$}
	\put(3835,-477){$n$}
	\end{picture}
 \caption{Parallel Q-nets $m$ and $m+s$ with the unit Gauss map $s$. All 
the nets are circular. The pair $(m,s)$ constitutes a principal contact 
element net.} \label{fig:circular=contact element}
 \end{figure}

The mean and the Gauss curvatures of the principal contact element nets 
$(m,s)$ are defined by formulas (\ref{eq:gaussmean}).

Proposition \ref{prop:quads} obviously implies: 

\begin{cor} For a circular quad mesh $m$, principal curvatures exist 
w.r.t.\ any Gauss image $s$ inscribed in $S^2$. \end{cor}

Recall also that circular Koenigs nets are identified in 
\cite{bobenko-2007-koenigs, bobenko-2008-ddg} as the discrete isothermic 
surfaces defined originally in \cite{bobenko-1996} as circular nets with 
factorizable cross-ratios.

Both minimal and constant mean curvature principal contact element nets 
are defined as in Section~\ref{s:constant curvature}. It is remarkable 
that the classes of circular minimal and cmc surfaces which are obtained 
via our definition of mean curvature turn out to be equivalent to the 
corresponding classes originally defined as special isothermic surfaces 
characterized by their Christoffel transformations. Since circular Koenigs 
nets are isothermic nets, from Theorem \ref{thm:minimal general} we 
recover the original definition of discrete minimal surfaces from 
\cite{bobenko-1996}.

\begin{cor} A principal contact element net 
$(m,s):\mathbb{Z}^2\to\mathbb{R}^3\times S^2$ is minimal if and only if 
the net $s:\mathbb{Z}^2\to S^2$ is isothermic and $m=s^*$ is its 
Christoffel dual. \end{cor}

Similarly, Theorem \ref{thm:CMC general} in the circular case implies that 
the discrete surfaces with constant mean curvature of 
\cite{Hertrich-Jeromin_Hoffmann_Pinkall_1999, BP99} fit into our 
framework.

\begin{cor}\label{cor:CMC principal contact element nets}
 A principal contact element net $(m,s):\mathbb{Z}^2\to\mathbb{R}^3\times 
S^2$ has constant mean curvature $H_0\neq 0$ if and only if the circular 
net $m$ is isothermic and there exists its dual discrete isothermic 
surface $m^*:\mathbb{Z}^2\to\mathbb{R}^3$ at constant distance 
$|m-m^*|=\frac{1}{H_0}$. The unit Gauss map $s$ which determines the 
principal contact element net $(m,s)$ is given by
	\begin{equation}\label{eq:n CMC}
	s=H_0(m^*-m).
	\end{equation}
 The principal contact element net of the parallel surface 
$(m+\frac{1}{H_0},-s)$ also has constant mean curvature $H_0$. The 
mid-surface $(m+\frac{1}{2H_0},s)$ has constant Gaussian curvature 
$4H_0^2$.
 \end{cor}

\begin{proof} Only the ``if'' part of the claim may require some 
additional consideration. If the discrete isothermic surfaces $m$ and 
$m^*$ are at constant distance $1/H_0$, then the map $s$ defined by 
(\ref{eq:n CMC}) maps into $S^2$ and is thus circular. Again, as in the 
proof of Theorem~\ref{thm:circular=contact element}, this implies that the 
contact element net $(m,s)$ is principal. Its mean curvature is given by 
	\begin{equation*} 
	-\frac{A(m,s)}{A(m,m)}=-\frac{A(m,H_0(m^*-m))}{A(m,m)}=H_0. 
	\end{equation*} 
 \end{proof}

\subsection{Minimal s-isothermic surfaces}

We now turn our attention to the discrete minimal surfaces $m$ of 
\cite{bobenko-2006-ms}, which arise by a Christoffel duality from a 
polyhedron $s$ which is midscribed to a sphere (a {\em Koebe polyhedron}). 
As Koebe polyhedra are up to M\"obius transformations determined by their 
combinatorics, a passage to the limit allows us to determine in this way 
the shape of smooth minimal surface from the combinatorics of the Gauss 
image of its network of principal curvature lines.

The Christoffel duality construction of \cite{bobenko-2006-ms} is applied 
to each face of $s$ separately. We consider a polygon $P=(p_0,\dots, 
p_{n-1})$ with $n$ even and incircle of radius $\rho$. We introduce the 
points $q_i$ where the edge $p_{i-1}p_i$ touches the incircle and identify 
the plane of $P$ with the complex numbers. In the notation of Figure 
\ref{fig:complex} the passage to the dual polygon $P^*$ is effected by 
changing the vectors $a_i=q_{2i}-z$, $b_i=q_{2i+1}-z$, $a'_i = 
p_{2i}-q_{2i+1}$, $b'_i=p_i-q_i$. Apart from multiplication with the 
factor $\pm\rho^2$, the corresponding vectors which define $P^*$ are given 
by
    \begin{equation}
    \label{eq:cr}
    a_j^*= (-1)^j /\Bar{a_j},\quad
    b_j^*=- (-1)^j /\Bar{b_j},\quad
    a_j^{\prime *}= (-1)^j/\Bar {a_j'},\quad
    b_j^{\prime *}=- (-1)^j /\Bar {b_j'}.
    \end{equation}
 The sign in the factor $\pm\rho^2$ depends on a certain labeling of 
vertices. The consistency of this construction and the passage to a 
branched covering in the case of odd $n$ is discussed in 
\cite{bobenko-2006-ms}. For us it is important that both $P$ and $P^*$ 
occur as concatenation of quadrilaterals:
    \begin{equation}
    P_j=(p_{j-1}q_j p_j q_j)\ \mbox{for}\
    j=0,\dots,n-1 \implies P=P_1\oplus\ldots\oplus P_{n-1},
    \end{equation}
 and the same for the starred (dual) entities. The main result
is the following:

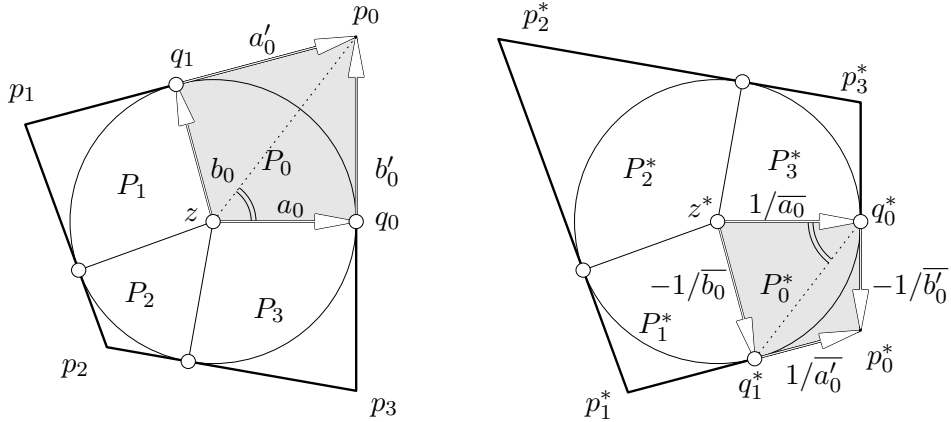
\begin{figure}[t]
    \medskip
\centerline{\begin{overpic}[height=.35\textwidth]{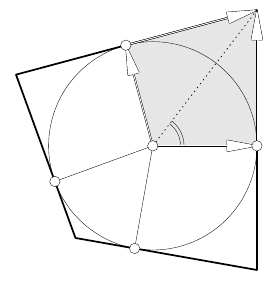}
    \put(70,51){$a_0$}
    \put(95,60){$b'_0$}
    \lput(50,49){$z$}
    \put(53,60){$b_0$}
    \lput(70,94){$a'_0$}
    \put(95,48){$q_0$}
    \cput(46,90){$q_1$}
    \cput(70,62){$P_0$}
    \cput(33,56){$P_1$}
    \cput(35,28){$P_2$}
    \cput(68,24){$P_3$}
    \cput(93,101){$p_0$}
    \cput(5,81){$p_1$}
    \lput(22,10){$p_2$}
    \put(94,0){$p_3$}
 \end{overpic}\qquad\qquad
 \begin{overpic}[height=.35\textwidth]{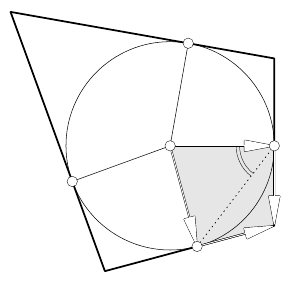}
    \cput(75,51){$1/\Bar{a_0}$}
    \put(99,50){$q_0^*$}
    \put(99,30){$-1/\Bar{b'_0}$} 
    \put(77,7){$1/\Bar{a'_0}$}
    \lput(62,30){$-1/\Bar{b_0}$}
    \lput(59,50){$z^*$}
    \put(98,12){$p_0^*$}
    \lput(33,0){$p_1^*$}
    \put(65,5){$q_1^*$}
    \put(10,100){$p_2^*$}
    \cput(95,84){$p_3^*$}
    \cput(77,62){$P_3^*$}
    \cput(40,60){$P_2^*$}
    \cput(44,20){$P_1^*$}
    \cput(75,30){$P_0^*$}
 \end{overpic}}
    \caption{Christoffel duality construction for
$s$\dash isothermic surfaces applied to a quadrilateral $P$ with incircle.
Corresponding sub\dash quadrilaterals $P_j$, $P_j^*$ have vanishing
mixed area.}
    \label{fig:complex}
\end{figure}

    \begin{thm} 
 A discrete s\dash isothermic minimal surface $m$ according to
\cite{bobenko-2006-ms} (Christoffel dual of a Koebe polyhedron $s$) has
vanishing mean curvature. Every face $f$ has principal curvatures
$\kappa_{1, f},\kappa_{2, f}=-\kappa_{1, f}$.
    \end{thm}

    \begin{proof} We first show that for all $j$, $A(P_j,P_j^*)=0$. This 
can be derived from \cite{bobenko-2006-ms} where it is shown that $P_j$ 
and $P_j^*$ are dual quads in the sense of discrete isothermic surfaces 
\cite{bobenko-1996}. Discrete isothermic surfaces are circular Koenigs 
nets \cite{bobenko-2008-ddg}, i.e., the quadrilaterals $P_j$ and $P_j^*$ 
are Christoffel dual in the sense of Definition \ref{def:dual}.

We can see this also in an elementary way which for $\rho=1$ is 
illustrated by Figure \ref{fig:complex}: The angle 
$\alpha_i=\sphericalangle(q_i, z, p_i)$ occurs also in the isosceles 
triangle $q_i^* z^* q_{i+1}^*$, so non\dash corresponding diagonals in 
$P_i,P_i^*$ are parallel. By Theorem \ref{thm: dual=orthogonal}, 
$A(P_i,P_i^*)=0$.

Lemma \ref{lem:concat} now implies that $A(P,P^*)=\sum A(P_j,P_j^*)=0$. 
Thus all faces of $m$ (i.e., the $P^*$'s of the previous discussion) have 
vanishing mixed area with respect to $s$. As the faces of $s$ are strictly 
convex, Prop.\ \ref{prop:quads} shows that principal curvatures exist.
 \end{proof}

\subsection{Discrete surfaces of rotational symmetry}

It is not difficult to impose the condition of constant mean or Gaussian 
curvature on discrete surfaces with rotational symmetry. In the following 
we briefly discuss this interesting class of examples.

We first consider quadrilateral meshes with regular grid combinatorics 
generated by iteratively applying a rotation about the $z$ axis to a {\em 
meridian polygon} contained in the $xz$ plane. Such surfaces have e.g.\ 
been considered by \cite{konopelchenko-1999-ts}.

The vertices of the meridian polygon are assumed to have coordinates 
$(r_i,0, h_i)$, where $i$ is the running index. The Gauss image of this 
polyhedral surface shall be generated in the same way, from the polygon 
with vertices $(r_i^*,0, h_i^*)$. Note that parallelity implies
	 \begin{equation}
	 {r_{i+1}-r_i\over
	h_{i+1}-h_i}={r_{i+1}^*-r_i^*\over h_{i+1}^*-h_i^*}.
	\label{eq:parallelity}
	 \end{equation}
 Figure \ref{fig:examples} illustrates such surfaces. All faces being 
trapezoids, it is elementary to compute mean and Gaussian curvatures 
$H^{(i)}$, $K^{(i)}$ of the faces bounded by the $i$-th and $(i+1)$-st 
parallel. It turns out that the angle of rotation is irrelevant for the 
curvatures:
    \begin{equation}
    H^{(i)} = {r_ir_i^*-r_{i+1}r_{i+1}^*\over
    r_{i+1}^2-r_i^2} , \quad K^{(i)} = {r_{i+1}^{*2}-r_i^{*2}\over
    r_{i+1}^2-r_i^2}. \label{eq:rot:HK} \end{equation}
 The principal curvatures associated with these faces have the values
    \begin{equation}
    \kappa_1^{(i)} = {r_{i+1}^*+r_i^* \over {r_{i+1}+r_i}}, \quad
    \kappa_2^{(i)} = {r_{i+1}^*-r_i^* \over {r_{i+1}-r_i}}.
    \label{eq:rot:princ}
    \end{equation}
 The interesting fact about these formulae is that the coordinates $h_i$ 
do not occur in them. Any functional relation involving the curvatures, 
and especially a constant value of any of the curvatures, leads to a 
difference equation for $(r_i)_{i\in\Z}$. For example, given an arbitrary 
Gauss image $(r_i^*,0, h_i^*)$ and the mean curvature function $H^{(i)}$ 
defined on the faces (which are canonically associated with the edges of 
the meridian curve) the values $r_i$ of the surface are determined by the 
difference equation (\ref{eq:rot:HK}) an an initial value $r_0$. Further 
the values $h_i$ follow from the parallelity condition 
(\ref{eq:parallelity}).

A meridian curve of a smooth surface of revolution does not intersect the 
rotation axis, and the Gauss map is spherical. Discrete analogues of such 
surfaces with a Gauss map $s\in S^2$ and prescribed curvature are 
determined by the values $(h^*_i)_{i\in\Z}$ lying in the interval 
$(-1,1)$, and an initial value $r_0$. The values 
$r^*_i=(1-{h^*_i}^2)^{1/2}$ should be chosen positive.

\begin{figure}[t]
	\centering
    \begin{overpic}[height=.21\linewidth]{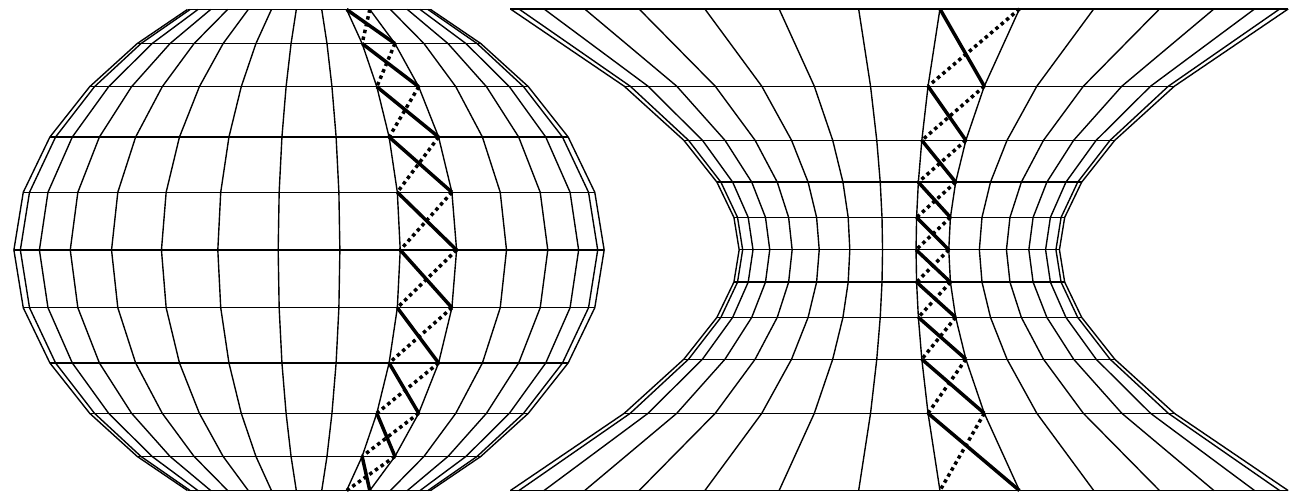}
        \lput(5,5){$s$}
        \put(85,17){$m$}
    \end{overpic}\hfill
    \begin{overpic}[height=.33\linewidth]{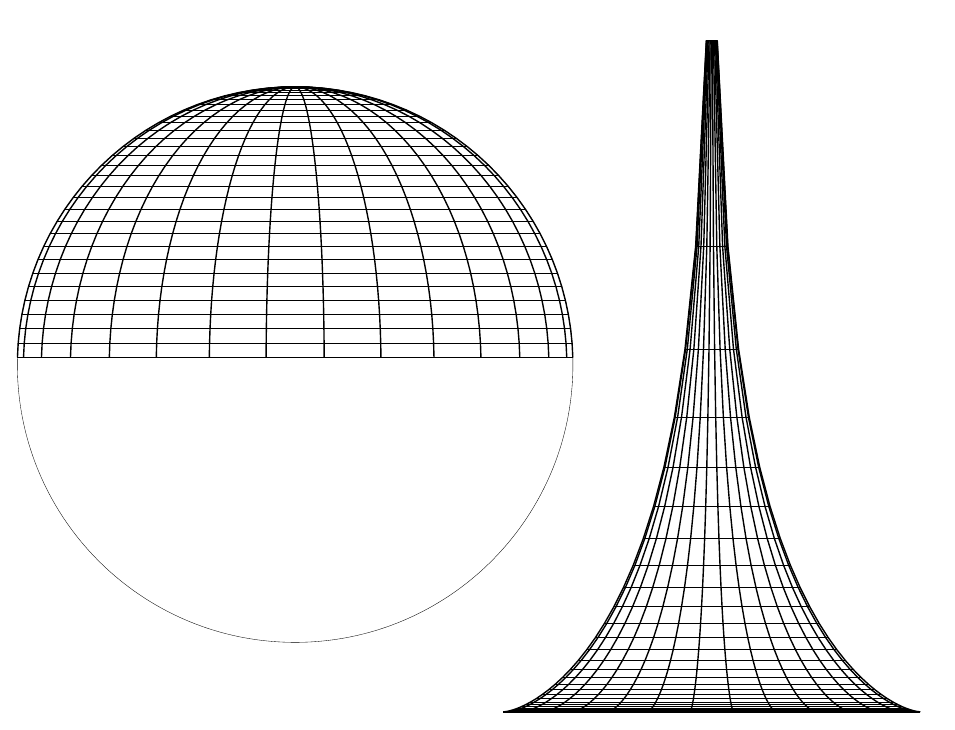}
       \lput(8,57){$s'$}
       \put(85,20){$m'$}
    \end{overpic}
    \caption{Left: A polyhedral surface $m$ which is a minimal surface 
w.r.t.\ to the Gaussian image $s$.  Right: A polyhedral surface $m'$ of 
constant Gaussian curvature w.r.t.\ the Gauss image $s'$ (discrete 
pseudosphere).}
    \label{fig:examples}
\end{figure}

 \begin{remark} The generation of a surface $m$ and its Gauss image $s$ by 
applying $k$-th powers of the same rotation to a meridian polygon 
(assuming axes of $m$ and $s$ are aligned) is a special case of applying a 
sequence of affine mappings, each of which leaves the axis fixed. It is 
easy to see that Equations \eqref{eq:rot:HK} and \eqref{eq:rot:princ} are 
true also in this more general case.
 \end{remark}

\begin{remark} While the formula for $\kappa_2$ given by 
\eqref{eq:rot:princ} is the usual definition of curvature for a planar 
curve, the formula for $\kappa_1$ can be interpreted as Meusnier's 
theorem. This is seen as follows: The curvature of the $i$-th parallel 
circle is given some average value of $1/r$ (in this case, the harmonic 
mean of $1/r_i$ and $1/r_{r+1}$). The sine of the angle $\alpha$ enclosed 
by the parallel's plane and the face under consideration is given by an 
average value of $r^*$ (this time, an arithmetic mean). By Meusnier, the 
normal curvature ``$\sin\alpha\cdot {1\over r}$'' of the parallel equals 
the principal curvature $\kappa_1$, in accordance with 
\eqref{eq:rot:princ}. \end{remark}

\begin{example} The mean curvature of faces given by \eqref{eq:rot:HK} 
vanishes if and only if $r_{i+1}:r_i = r_i^*:r_{i+1}^*$. This condition is 
converted into the first order difference equation
    \begin{equation}
    \Delta\ln r_i = - \Delta \ln r^*_i \quad (i\in\Z),
    \label{eq:catenoid}
    \end{equation}
 where $\Delta$ is the forward difference operator. It is not difficult to 
see that the corresponding differential equation $(\ln r)'=-(\ln r^*)'$ is 
fulfilled by the catenoid: With the meridian $(t,\cosh t)$ and the unit 
normal vector $(-\tanh t, 1/\cosh t)$ we have $r(t)=\cosh t$ and 
$r^*(t)=1/\cosh t$. We therefore like to denote discrete surfaces 
fulfilling \eqref{eq:catenoid} discrete catenoids (see Figure 
\ref{fig:examples}, left).
 \end{example}

 \begin{example} A discrete surface of constant Gaussian curvature $K$ 
obeys the difference equation $K \Delta (r_i^2) = \Delta (r_i^{*2})$. 
Figure \ref{fig:examples}, right illustrates a solutions.
 \end{example}

\subsection{Discrete surfaces of rotational symmetry with constant mean 
curvature and elliptic billiards}

There exists a nice geometric construction of discrete surfaces of 
rotational symmetry with constant mean curvature, which we obtained 
jointly with Tim Hoffmann. This is a discrete version of the classical 
Delaunay rolling ellipse construction for surfaces of revolution with 
constant mean curvature (Delaunay surfaces).

Play an extrinsic billiard around an ellipse $E$. A trajectory is a 
polygonal curve $P_1,P_2,\ldots$ such that the intervals $[P_i,P_{i+1}]$ 
touch the ellipse $E$ and consecutive triples of vertices 
$P_{i-1},P_i,P_{i+1}$ are not collinear (see Figure \ref{fig:billiard}). 
Let us connect the vertices $P_i$ to the focal point $B$, and roll the 
trajectory $P_1,P_2,\ldots$ to a straight line $\ell$, mapping the 
triangles $BP_iP_{i+1}$ of Figure \ref{fig:billiard} isometrically to the 
triangles $B_iP_iP_{i+1}$ of Figure \ref{fig:rolling}. We use the same 
notations for the vertices of the billiard trajectory and their images on 
the straight line, and the points $B_i$ are chosen in the same half-plane 
of $\ell$. Thus we have constructed a polygonal curve $B_1,B_2,\ldots$ . 
Applying the same construction to the second focal point $A$ we obtain 
another polygonal curve $A_1,A_2,\ldots$, chosen to lie in another 
half-plane of $\ell$.

\begin{figure}[t]
	\setlength{\unitlength}{3552sp}
\begin{picture}(3429,4336)(5244,-4520)
	\put(5244,-4520){\includegraphics[width=3860\unitlength]
		{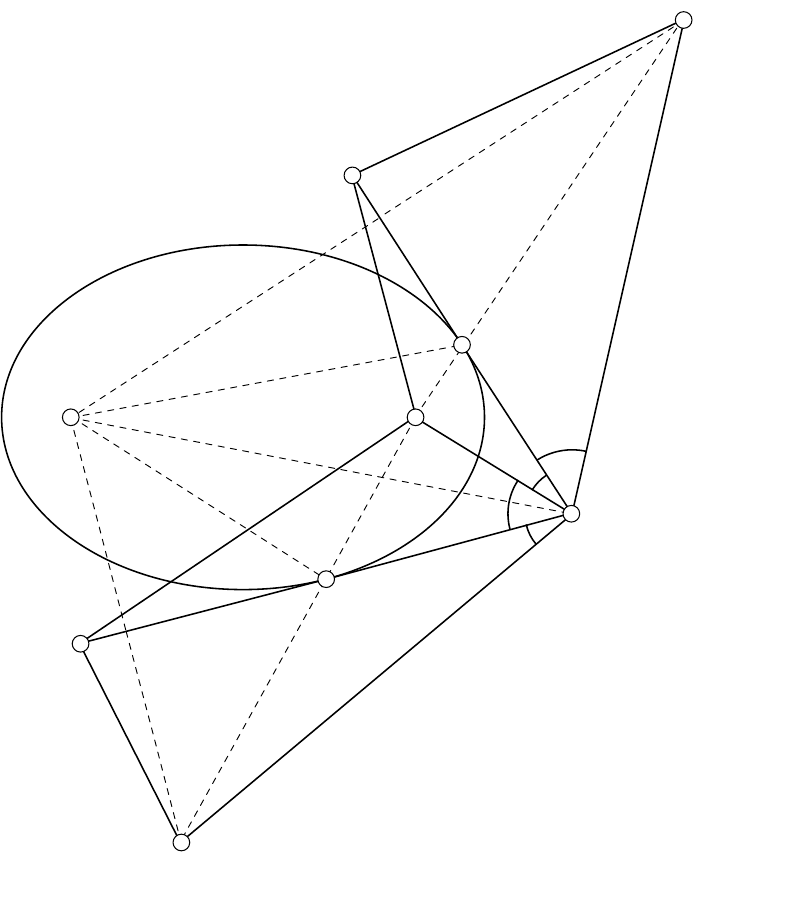}}
	\put(8658,-319){$A_2$}
	\put(6701,-946){$P_3$}
	\put(8074,-2773){$P_2$}
	\put(7880,-2255){$\gamma$}
	\put(7630,-2377){$\beta$}
	\put(7465,-2703){$\gamma$}
	\put(7529,-2916){$\beta$}
	\put(5800,-4300){$A_1$}
	\put(5474,-2083){$A$}
	\put(6978,-2187){$B$}
	\put(5363,-3326){$P_1$}
\end{picture}
	\vskip-1em
	\caption{An external elliptic billiard. The trajectory
	$\{P_i\}_{i\in\Z}$ is tangent to an ellipse.}
    \label{fig:billiard}
	\setlength{\unitlength}{3552sp}
\begin{picture}(4682,2845)(1087,-5494)
	\put(1087,-5494){\includegraphics[width=5060\unitlength]
		{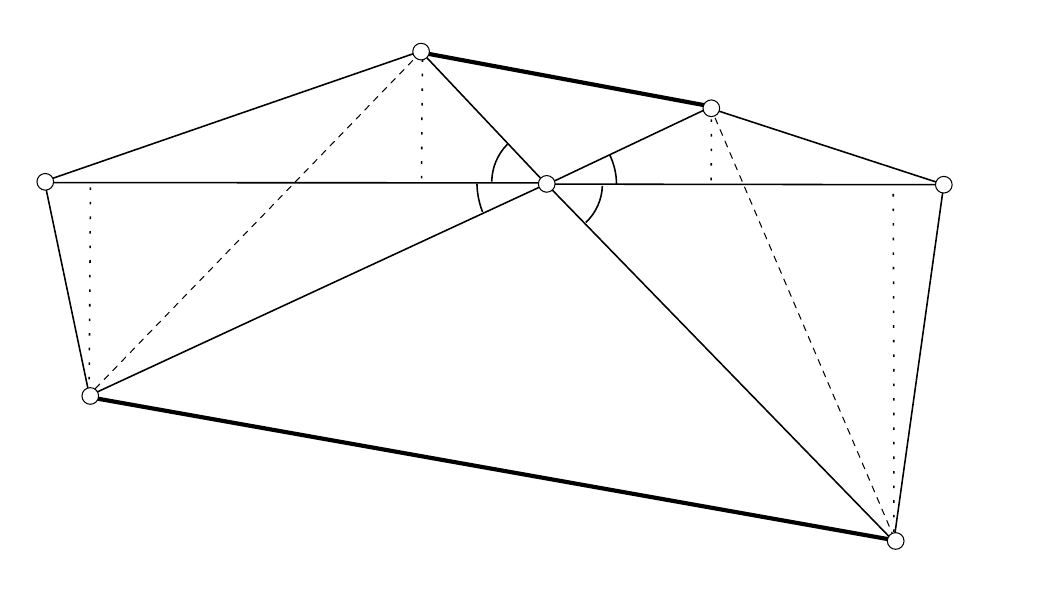}}
	\put(4526,-3070){$B_2$}
	\put(3103,-2784){$B_1$}
	\put(5754,-3575){$P_3$}
	\put(3335,-3438){$\gamma$}
	\put(1592,-3942){$r_1'$}
	\cput(2958,-3364){$r_1$}
	\cput(1102,-3561){$P_1$}
	\put(1376,-4800){$A_1$}
	\put(3212,-3691){$\beta$}
	\put(3628,-3763){$P_2$}
	\put(4022,-3741){$\gamma$}
	\put(3826,-3412){$\beta$}
	\put(5520,-5200){$A_2$}
	\cput(4343,-3425){$r_2$}
	\cput(5204,-3994){$r_2'$}
\end{picture}%
	\vskip-1em
\caption{A discrete cmc surface with rotational
 symmetry generated from an elliptic billiard.}
    \label{fig:rolling}
\end{figure}

Let us consider discrete surfaces $m$ and $\wt m$ with rotational symmetry 
axis $\ell$ generated by the meridian polygons constructed above: 
$m_i=B_i$, $\wt m_i=A_i$. They are circular surfaces which one can provide 
with the same Gauss map $s_i:=m_i-\wt m_i$.

 \begin{thm} \label{thm:cmc_rolling}
 Let $P_1,P_2,\ldots$ be a trajectory of an extrinsic elliptic billiard 
with the focal points $A,B$. Let $m,\wt m$ be the circular surfaces with 
rotational symmetry generated by the discrete rolling ellipse construction 
in Figures \ref{fig:billiard}, \ref{fig:rolling}: $m_i=B_i$, $\wt 
m_i=A_i$. Both surfaces $(m,s)$ and $(\wt m,-s)$ with the Gauss map 
$s=m-\wt m$ have constant mean curvature $H$, where $1/H=|A_1B|$ equals 
twice the major axis of the ellipse (see Figure~\ref{fig:billiard}).
 \end{thm}

 \begin{proof} The sum of the distances from a point of an ellipse to the 
focal points is independent of the point, i.e.,
	$$
	l:=| A_iB_i|
	$$
 is independent of $i$. Due to the equal angle lemma of Figure 
\ref{fig:equal_angles} we have equal angles $\beta:=\angle 
P_1P_2A_1=\angle BP_2P_3$ and $\gamma:=\angle P_1P_2B=\angle P_3P_2A_2$ in 
Figure \ref{fig:billiard}. Thus $P_2$ in Figure \ref{fig:rolling} is the 
intersection point of the straight lines $(A_1B_2)\cap (B_1A_2)$. Similar 
triangles $\triangle P_2A_1A_2\sim \triangle P_2B_2B_1$ imply parallel 
edges $(A_1A_2)\parallel (B_1B_2)$. This yields the proportionality 
$r_i/r_{i+1}=r'_{i+1}/r'_i$ for the distances $r$ to the axis $\ell$. For 
the mean curvature of the surface $m$ with the Gauss image $s=m-\wt m$ we 
obtain from (\ref{eq:rot:HK}):
	$$
	H=\frac{1}{l}
	\frac{r_i(r'_i-r_i)-r_{i+1}(r'_{i+1}-r_{i+1})}
		{r_{i+1}^2-r_i^2}
	=\frac{1}{l}.
	$$
 The surface $\wt m$ is the parallel cmc surface of Corollary \ref{cor:CMC 
principal contact element nets}.
 \end{proof}

 \begin{figure}
 \includegraphics{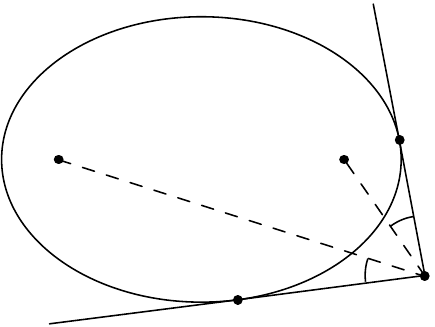}
 \caption{The angles between the tangent directions and the
directions to the focal points of an ellipse are equal.}
    \label{fig:equal_angles}
 \end{figure}

If the vertices of the trajectory $P_1,P_2,\ldots$ lie on an
ellipse $E'$ confocal with $E$, then it is a classical reflection
billiard in the ellipse $E'$ (see for example
\cite{tabachnikov-2005}). The sum
	$$
	d:=| AP_i|+| BP_i|
	$$
 is independent of $i$. The quadrilaterals $A_iA_{i+1}B_{i+1}B_i$ in 
Figure \ref{fig:rolling} have equal diagonals, i.e., are trapezoids. The 
product of the lengths of their parallel edges is independent of $i$:
	\begin{equation}\label{eq:d_l}
	|A_iA_{i+1}||B_iB_{i+1}|=d^2-l^2.
	\end{equation}
 As we have shown in the proof of Theorem \ref{thm:cmc_rolling}, $r_ir'_i$ 
is another product independent of $i$. An elementary computation gives the 
same result for the cross-ratios of a faces of the discrete surfaces $m$ 
and $\wt m$:
	$$
	q=-\frac{1}{\sin^2
	\alpha}\frac{|A_iA_{i+1}||B_iB_{i+1}|}{r_ir'_i},
	$$
 where $2\alpha$ is the rotation symmetry angle of the surface. We see 
that $q$ is the same for all faces of the surfaces $m$ and $\wt m$.

We have derived the main result of \cite{hoffmann-1998}.

 \begin{cor}
 Let $P_1,P_2,\ldots$ be a trajectory of a classical reflection elliptic 
billiard, and $m,\wt m$ be the discrete surfaces with rotational symmetry 
generated by the discrete rolling ellipse construction as in Theorem 
\ref{thm:cmc_rolling}. Both these surfaces have constant mean curvature 
and constant cross-ratio of their faces.
 \end{cor}

The discrete rolling construction applied to hyperbolic billiards also 
generates discrete cmc surfaces with rotational symmetry.

\section{Concluding remarks}

We would like to mention some topics of future research. We have treated 
curvatures of faces and of edges. It would be desirable to extend the 
developed theory to define curvature also at vertices. A large area of 
research is to extend the present theory to the semidiscrete surfaces 
which have recently found attention in the geometry processing community, 
and where initial results have already been obtained.

\section*{Acknowledgments}

This research was supported by grants P19214-N18, S92-06, and S92-09 of 
the Austrian Science Foundation (FWF), and by the DFG Research Unit 
``Polyhedral Surfaces''.

\def\MR#1{\relax}

% amsplain1.bst uses initials instead of full first names
%\bibliographystyle{amsplain1}
%\bibliography{pkmesh}

\def\http#1{{\it\spaceskip 0 pt plus 0.5pt http:/$\!$/ #1}}
  \hyphenation{Spring-born Hoff-mann Pin-kall} \hyphenation{Sau-er Blasch-ke
  Pol-thier War-detz-ky}  \hyphenation{Wall-ner Pott-mann Wun-der-lich}
  \def\Yu{Yu} \hyphenation{Dif-fe-ren-zen-geo-met-rie}
  \hyphenation{Ha-bi-li-ta-ti-ons-schrift} \hyphenation{La-guerre Bla-schitz}
\providecommand{\bysame}{\leavevmode\hbox to3em{\hrulefill}\thinspace}
\providecommand{\MR}{\relax\ifhmode\unskip\space\fi MR }
% \MRhref is called by the amsart/book/proc definition of \MR.
\providecommand{\MRhref}[2]{%
  \href{http://www.ams.org/mathscinet-getitem?mr=#1}{#2}
}
\providecommand{\href}[2]{#2}
\begin{thebibliography}{10}

\bibitem{bobenko-2006-ms}
A.~I. Bobenko, T.~Hoffmann, and B.~Springborn, \emph{Minimal surfaces from
  circle patterns: Geometry from combinatorics}, Ann. of Math. \textbf{164}
  (2006), 231--264. \MR{2007b:53006}

\bibitem{Bobenko_Schroeder_Sullivan_Ziegler_2008}
A.~I. Bobenko, S.~P., J.~M. Sullivan, and G.~M. Ziegler (eds.), \emph{Discrete
  differential geometry}, Oberwolfach Seminars, vol.~38, Birkh\"auser, Basel,
  2008.

\bibitem{bobenko-1996}
A.~I. Bobenko and U.~Pinkall, \emph{Discrete isothermic surfaces}, J. Reine
  Angew. Math. \textbf{475} (1996), 187--208. \MR{97f:53004}

\bibitem{BP99}
A.~I. Bobenko and U.~Pinkall, \emph{Discretization of surfaces and integrable
  systems}, Discrete integrable geometry and physics (A.~I. Bobenko and
  R.~Seiler, eds.), Oxford Lecture Ser. Math. Appl., vol.~16, Oxford Univ.\
  Press, 1999, pp.~3--58. \MR{2001j:37128}

\bibitem{bobenko-2007-org}
A.~I. Bobenko and {\Yu}.~Suris, \emph{On organizing principles of discrete
  differential geometry. {G}eometry of spheres}, Russian Math. Surveys
  \textbf{62} (2007), no.~1, 1--43.

\bibitem{bobenko-2008-ddg}
\bysame, \emph{Discrete differential geometry. {I}ntegrable structure},
  Graduate Studies in Math., no.~98, American Math. Soc., 2008.

\bibitem{bobenko-2007-koenigs}
\bysame, \emph{Discrete {K}oenigs nets and discrete isothermic surfaces}, Int.
  Math. Res. Not. (2009), to appear.

\bibitem{christoffel}
E.~Christoffel, \emph{Ueber einige allgemeine {E}igenschaften der
  {M}inimumsfl\"achen}, J. Reine Angew. Math. \textbf{67} (1867), 218--228.

\bibitem{Fordy_Wood_1994}
A.~P. Fordy and J.~C. Wood (eds.), \emph{Harmonic maps and integrable systems},
  Aspects of Mathematics, vol. E23, Vieweg, Braunschweig, 1994.

\bibitem{Hertrich-Jeromin_Hoffmann_Pinkall_1999}
U.~Hertrich-Jeromin, T.~Hoffmann, and U.~Pinkall, \emph{A discrete version of
  the {D}arboux transform for isothermic surfaces}, Discrete integrable
  geometry and physics (A.~I. Bobenko and R.~Seiler, eds.), Clarendon Press,
  Oxford, 1999, pp.~59--81.

\bibitem{hoffmann-1998}
T.~Hoffmann, \emph{Discrete rotational cmc surfaces and the elliptic billiard},
  Mathematical Visualization (H.-C. Hege and K.~Polthier, eds.), Springer,
  Berlin, 1998, pp.~117--124.

\bibitem{konopelchenko-1999-ts}
B.~G. Konopelchenko and W.~K. Schief, \emph{Trapezoidal discrete surfaces:
  geometry and integrability}, J.\ Geometry Physics \textbf{31} (1999), 75--95.

\bibitem{Pinkall_Polthier_1993}
U.~Pinkall and K.~Polthier, \emph{Computing discrete minimal surfaces and their
  conjugates}, Experiment. Math. \textbf{2} (1993), no.~1, 15--36.

\bibitem{pottmann-2008-eolag}
H.~Pottmann, P.~Grohs, and B.~Blaschitz, \emph{Edge offset meshes in {L}aguerre
  geometry}, Adv. Comput. Math. (2009), to appear.

\bibitem{pottmann-2007-pm}
H.~Pottmann, Y.~Liu, J.~Wallner, A.~I. Bobenko, and W.~Wang, \emph{Geometry of
  multi-layer freeform structures for architecture}, ACM Trans. Graphics
  \textbf{26} (2007), no.~3, \#65, 11 pp.

\bibitem{pottmann-2008-fg}
H.~Pottmann and J.~Wallner, \emph{The focal geometry of circular and conical
  meshes}, Adv. Comp. Math \textbf{29} (2008), 249--268.

\bibitem{Rogers_Schief_2002}
C.~Rogers and W.~K. Schief, \emph{{B\"acklund and {D}arboux transformations.
  Geometry and modern applications in soliton theory}}, Cambridge Texts in
  Applied Mathematics, Cambridge University Press, Cambridge, 2002.

\bibitem{Sauer_1950}
R.~Sauer, \emph{{Parallelogrammgitter als Modelle pseudosph\"arischer
  Fl\"achen}}, Math. Z. \textbf{52} (1950), 611--622.

\bibitem{Schief_2003}
W.~K. Schief, \emph{On the unification of classical and novel integrable
  surfaces. {II}. {D}ifference geometry}, R. Soc. Lond. Proc. Ser. A
  \textbf{459} (2003), 373--391.

\bibitem{schief-2006}
W.~K. Schief, \emph{On a maximum principle for minimal surfaces and their
  integrable discrete counterparts}, J. Geom. Physics \textbf{56} (2006),
  1484--1495. \MR{2007f:53006}

\bibitem{schneider-1993}
R.~Schneider, \emph{Convex bodies: the {B}runn\dash {M}inkowski theory},
  Encyclopedia of Mathematics and its Applications, vol.~44, Cambridge
  University Press, 1993. \MR{94d:52007}

\bibitem{simon:1992}
U.~Simon, A.~Schwenck{\dash}Schellschmidt, and H.~Viesel, \emph{Introduction to
  the affine differential geometry of hypersurfaces. {L}ecture notes}, Science
  Univ.\ Tokyo, 1992. \MR{94b:53023}

\bibitem{tabachnikov-2005}
S.~Tabachnikov, \emph{Geometry and billiards}, Student Mathematical Library,
  no.~30, American Math. Soc., 2005.

\bibitem{Wunderlich_1951}
W.~Wunderlich, \emph{{Zur Differenzengeometrie der Fl{\"a}chen konstanter
  negativer Kr{\"u}mmung}}, Sitz.\ {\"O}st.. Akad. Wiss. Math.-Nat. Kl.
  \textbf{160} (1951), 39--77.

\end{thebibliography}

\def\http#1{{\it\spaceskip 0 pt plus 0.5pt http:/$\!$/ #1}}
  \hyphenation{Spring-born Hoff-mann Pin-kall} \hyphenation{Sau-er Blasch-ke
  Pol-thier War-detz-ky}  \hyphenation{Wall-ner Pott-mann Wun-der-lich}
  \def\Yu{Yu} \hyphenation{Dif-fe-ren-zen-geo-met-rie}
  \hyphenation{Ha-bi-li-ta-ti-ons-schrift} \hyphenation{La-guerre Bla-schitz}
\providecommand{\bysame}{\leavevmode\hbox to3em{\hrulefill}\thinspace}
\providecommand{\MR}{\relax\ifhmode\unskip\space\fi MR }
% \MRhref is called by the amsart/book/proc definition of \MR.
\providecommand{\MRhref}[2]{%
  \href{http://www.ams.org/mathscinet-getitem?mr=#1}{#2}
}
\providecommand{\href}[2]{#2}

\end{document}